\documentclass{amsart}
\usepackage{amsthm,amssymb,latexsym,amsmath, ae}
\usepackage{tikz}
\usepackage{enumerate,url}

\def \C {\mathbb{C}}

\newcommand{\p}{\Bbb{P}}

\def \F {{\mathcal F}}
\newcommand{\OO}{\mathcal{O}}

\def \R {{\mathbb R}}

\newtheorem{proposition}{Proposition}[section]
\newtheorem{definition}[proposition]{Definition}
\newtheorem{corollary}[proposition]{Corollary}
\newtheorem{lemma}[proposition]{Lemma}
\newtheorem{theorem}[proposition]{Theorem}

\newtheorem{remark}[proposition]{Remark}

\begin{document}

\title[Poincar\'e problem for weighted projective foliations]{Poincar\'e problem for weighted projective foliations}

\author{F. E. Brochero Mart\'inez}

\address{F. E. Brochero Mart\'inez \\
Departamento de Matem\'atica \\
Universidade Federal de Minas Gerais\\
Av. Ant\^onio Carlos 6627 \\
30123-970 Belo Horizonte MG, Brazil} \email{fbrocher@mat.ufmg.br}

\author{M.  Corr\^ea Jr. }
\thanks{Partially supported by CNPq grant number 300352/2012-3 and FAPEMIG grant number PPM-00169-13.}
\address{\noindent Maur\' \i cio Corr\^ea Jr\\
Departamento de Matem\'atica \\
Universidade Federal de Minas Gerais\\
Av. Ant\^onio Carlos 6627 \\
30123-970 Belo Horizonte MG, Brazil} \email{mauricio@mat.ufmg.br}

\author{A. M. Rodr\'iguez  }

\address{A. M. Rodr\'iguez  \\
Departamento de Matem\'atica \\
Universidade Federal de Minas Gerais\\
Av. Ant\^onio Carlos 6627 \\
30123-970 Belo Horizonte MG, Brazil} \email{miguel.rodriguez.mat@gmail.com}


\begin{abstract}
We give a  bounding of  degree of  quasi-smooth hypersurfaces  which are invariant
by a one dimensional holomorphic foliation of a given degree on a
weighted projective space.
\end{abstract}
\maketitle

\section{Introduction}

Henri  Poincar\'e  studied in \cite{Po} the problem  to decide if a holomorphic
foliation $\F$ on the complex projective plane
$\mathbb{P}^2$ admits a rational first integral. 
Poincar\'e  observed that in order to solve this problem is sufficient
to find a bound for the degree of the generic curve  invariant by $\F$. In
general, this is not possible, but doing some hypothesis we obtain
an affirmative answer for this problem, which nowadays is known as
$\emph{Poincar\'e  Problem}$. This problem was treated by D. Cerveau and A. Lins Neto \cite{C-LN}.
M. Brunella in \cite{B1} observed that obstruction  to the
positive solution to Poincar\' e problem is given by the GSV index. There exist several works about Poincar\' e problem and its generalizations, see for instance the papers:    M. Carnicer \cite{Ca}, J. V. Pereira  \cite{Pe} , M . Brunella and L.G. Mendes \cite{BM}, E. Esteves and S. Kleiman \cite{EK}.

M. Soares in  \cite{So} proved the following Theorem for smooth hypersurfaces invariant by foliations on $\mathbb{P}^ n$ .

\begin{theorem} \cite{So} 
Let $\mathcal{F}$ be a holomorphic one dimensional  foliation   on
$\mathbb{P}^ n$  with
isolated singularities. If  $V \subset  \mathbb{P}^ n$ is  a
smooth hypersurface invariant by  $\mathcal{F}$, then 
$$
\deg(V)\leq \deg(\mathcal{F})+1.
$$
\end{theorem}

In \cite{CS} M. Corr\^ea Jr and M. Soares studied the Poincar\'e problem for foliations on weighted projective planes. 

\begin{theorem}\label{Mauricio-Marcio}
Let $\F$ be a foliation on  the weighted projective plane $\mathbb P(\omega_0,\omega_1,\omega_2)$ such that $Sing(\F)\cap
Sing(\mathbb P(\omega_0,\omega_1,\omega_2))=\emptyset$ . If  $S$ is a  quasi-smooth invariant  curve, then 
$$\deg(S)\leq
\deg(\F)+\omega_0+\omega_1+\omega_2 -2.$$
\end{theorem}

In this work, we give a  bounding of  degree of  quasi-smooth hypersurfaces  which are invariant
by a one dimensional holomorphic foliation of a given degree on a
weighted projective space.
Throughout this paper, $\mathbb P(\omega)$ will denote the weighted projective space of dimension $n$
and weights $(\omega_0,\dots,\omega_n)$.

We prove the following theorem.

\begin{theorem} \label{Theopri}
Let $\mathcal{F}$ be a holomorphic one dimensional  foliation on
$\mathbb{P}(\omega)$ with isolated singularities, let $V \subset  \mathbb{P}(\omega)$ be a
quasi-smooth hypersurface invariant by $\mathcal{F}$.
\begin{enumerate}[(i)]
\item If $n=2$, then
 $$
\deg(V)\leq \deg(\mathcal{F})+\omega_0+\omega_1+\omega_2 -2;
$$
\item if
 $n\geq 3$ and $\deg(\mathcal{F}) \ge \omega_0+\cdots+\omega_n+1$, then 
$$
\deg(V)< \deg(\mathcal{F})+\alpha_n(\omega_0+\cdots+\omega_n) -1,
$$%where $\deg(V)$ is the degree of quasi-homogeneity of $V$.
where $\alpha_n=\begin{cases}\text{the positive root of $R_n(x):=x(x+1)^n-2=0$}& \text{if $n$ is odd}\\
\alpha_{n-1}&\text{if $n$ is even}\end{cases}$
\end{enumerate}
\end{theorem}
This result say us that the  hypothesis $Sing(\F)\cap
Sing(\mathbb P(\omega))=\emptyset$ is not necessary. Therefore, the Theorem \ref{Mauricio-Marcio} is improved  in the case $n=2$ and generalized  whenever
 the condition $\deg(\mathcal{F}) \ge \omega_0+\cdots+\omega_n+1$ holds.

\begin{remark} By a direct computation,  we can calculate the first values of $\alpha_n$ with 4 decimal places:
\begin{center}
\begin{tabular}{|c|c|c|c|c|c|c|c|c|c|}\hline
$ n$&$3$&$5$&$7$&$9$&$11$&$13 $&$15$&$17$&$19$\\ \hline
$\alpha_n$&\scriptsize$0.5436$&\scriptsize$0.3880$&\scriptsize$0.3069$&\scriptsize$0.2563$&\scriptsize$0.2214$&\scriptsize$0.1957$&\scriptsize$0.1759$&\scriptsize$0.1601$&\scriptsize$0.1471$
\\ \hline
\end{tabular}
\end{center}
%\begin{center}
%\begin{tabular}{|c|c|c|c|c|c|c|c|c|c|c|c|c|}\hline
%$ n$&$3$&$5$&$7$&$9$&$11$&$13 $&$15$&$17$&$19$&$21$&$23$&$25$\\ \hline
%$\alpha_n$&\scriptsize$0.543$&\scriptsize$0.388$&\scriptsize$0.306$&\scriptsize$0.256$&\scriptsize$0.221$&\scriptsize$0.195$&\scriptsize$0.175$&\scriptsize$0.160$&\scriptsize$0.147$&\scriptsize$0.136$&\scriptsize$0.127$&\scriptsize$0.119$
%\\ \hline
%\end{tabular}
%\end{center}
%\begin{center}
%
%\begin{tabular}{|c|c|c|c|c|c|c|c|c|c|c|c|c|}\hline
%n&3&5&7&9&11&13&15&17&19&21&23&25\\ \hline
%$\alpha_n$&0.54368&0.38809&0.30698&0.25639&0.22147&0.19575&0.17592&0.16012&0.14719&0.13640&0.12724&0.11935
%\\ \hline
%\end{tabular}
%\end{center}
%we can see that $0.543688<\alpha_3<0.54369$,  $0.38809<\alpha_5<0.3881$ and $0.30699<\alpha_7<0.307$.  In addition, 
%$$\frac{\ln \epsilon n}{n}\left(1+\frac{\ln \epsilon n}{n}\right)^n
%P_n\left(\frac{\ln \epsilon n}{n}\right)
%\approx \frac{\ln \epsilon n}{n} \exp(\ln \epsilon n)-2=\epsilon \ln \epsilon n-2\gg 0%\to\infty
%,$$
%and

In addition, since
$$%\frac{\ln n -\ln\ln n}{n}\left(1+\frac{\ln n-\ln\ln n}{n}\right)^n
R_n\left(\frac{\ln n -\ln\ln n}{n}\right)< \frac{\ln n-\ln\ln n}{n} \exp(\ln n-\ln\ln n)-2=\frac {\ln n-\ln\ln n}{\ln n}-2 <-1%\to 0
$$
and for any constant $\epsilon>0$, and all $n\gg 0$  
\begin{align*}
R_n\left(\frac{\ln n -(1-\epsilon)\ln\ln n}{n}\right)&\approx \frac{\ln n-(1-\epsilon)\ln\ln n}{n} \exp(\ln n-(1-\epsilon)\ln\ln n)-2\\
&=(\ln n)^\epsilon\left(1-\frac {\ln n-\ln\ln n}{\ln n}\right)-2 \gg 0.%\to 0
\end{align*}
Using that $R_n(x)$ is an increasing function in $\R^+$, it follows that $$\frac{\ln n -\ln\ln n}{n}<\alpha_n<\frac{\ln n -(1-\epsilon)\ln\ln n}{n}\quad \text{ for all $n\gg 0$.}$$
%using   Bachmann-Landau notations, $\alpha_n=o\left(\frac {\ln n}{n}\right)$ and $\alpha_n=\omega\left(\frac{\ln n -\ln\ln n}{n}\right)$. 
 In general, from the fact that  $\frac 2{n+1}< \frac{\ln n -\ln\ln n}{n}$ for all $n\ge 21$, we have that $\max\left\{\frac {2}{n+1}, \frac{\ln n -\ln\ln n}{n}\right\}<\alpha_n<\frac {\ln 2n}n$ for all $n\ge 3$.
\end{remark}

Let us give a family of examples of  holomorphic foliations satisfying the conditions of Theorem \ref{Theopri}:

Let $a_0,b_0,a_1,\dots, a_n,b_n$ be positive integers, without common factor in pairs and such that
 $$\xi:= a_{0}+b_{0} =\dots= a_{n}+b_{n}%\ge2n+2
. $$
and consider the  well formed   weighted projective space   $\mathbb{P}^{2n+1}(a_{0},b_{0},\dots,a_{n},b_{n})$. 

Let $\F$ be the  holomorphic foliation  on  $\mathbb{P}^{2n+1}(a_{0},b_{0},\dots,a_{n},b_{n})$, induced by the quasi-homogeneous vector field 
$$Z=\sum_{k=0}^{n}\Big(\beta_{k}Y_{k}^{\beta_{k}-1}\frac{\partial}{\partial X_{k}}-\alpha_{k}X_{k}^{\alpha_{k}-1}\frac{\partial}{\partial Y_{k}}\Big),$$
where the $\alpha_{k}, \beta_{k} \in \mathbb{N}$ satisfy the following relation
$$\zeta=a_{k}\alpha_{k}=b_{k}\beta_{k}\qquad \text{for all}\  k=0,\dots,n.$$
	
A quasi-smooth hypersurface on $ \mathbb{P}^{2n+1}(a_{0},b_{0},\dots,a_{n},b_{n})$ of degree $\zeta$ given by 
$$
V=\left\{\sum_{k=0}^{n}\big(X_{k}^{\alpha_{k}}+Y_{k}^{\beta_{k}}\big)=0\right\}.
$$
We can see that $V$ is invariant by $\F$  and $\deg (\F) = \zeta - \xi + 1$ . Moreover, 
   since $a_i$ and $b_i$ divide $\zeta$, it follows that $\zeta\ge a_0b_0\cdots a_nb_n\gg (n+2)\xi$ and
$$\deg (\F) = \zeta - \xi + 1 \geq (n+1)\xi+1=1+\sum_{j=0}^n a_j+b_j.$$
So,  the hypothesis of Theorem \ref{Theopri} is satisfied. 

Finally, %since $\alpha_{2n+1}>\frac 1{n+1}$,
 we have 
$$\deg (V)-\deg(\F)= \xi -1 =\frac 1{n+1}\Bigl(\sum_{j=0}^n a_j+b_j \Bigr)- 1<\alpha_{2n+1}\Bigl(\sum_{j=0}^n a_j+b_j \Bigr) - 1.$$
We can construct a similar foliation on even dimensional  weighted projective spaces $\mathbb{P}^{2n+2}(a_{0},b_{0},\dots,a_{n},b_{n},a_{n+1})$  where $\xi=a_{k}+b_{k}\ \text{for all }\ k=0,\dots,n.$

Let suppose that $\zeta=a_{k}\alpha_{k}=b_{k}\beta_{k}=a_{n+1}\alpha_{n+1}$ for all $\, k=0,\dots,n$ and consider
the vector field $Z$ in the previous example. Thus, the quasi-smooth hypersurface on $ \mathbb{P}^{2n+2}(a_{0},b_{0},\dots,a_{n},b_{n},a_{n+1})$ of degree $\zeta$ given by 
$$
V=\left\{\sum_{k=0}^{n}\big(X_{k}^{\alpha_{k}}+Y_{k}^{\beta_{k}}\big)+X_{n+1}^{\alpha_{n+1}}=0\right\}
$$
is invariant by $Z$ and therefore we obtain the same conclusions.

\section{Weighted projective foliations}

Let $\omega_0,\dots,\omega_n$ be integers $\geq 1$. Consider the
$\mathbb{C}^*$-action on $\mathbb{C}^{n+1}\backslash \{0\}$ given by
$$
\lambda\cdot(z_0,\dots,z_n)=(\lambda^{\omega_0}
z_0,\dots,\lambda^{\omega_n} z_n),
$$
where $\lambda\in \mathbb{C}^*$ and $(z_0,\dots,z_n)\in
\mathbb{C}^{n+1}\backslash \{0\}$. The \emph{weighted projective
space of type $(\omega_0,\dots,\omega_n)$} is the quotient space $
\mathbb{P}(\omega_0,\dots,\omega_n)=(\mathbb{C}^{n+1}\setminus
\{0\}/\sim) $, induced by the action above . We will abbreviate
$\mathbb{P}(\omega_0,\dots,\omega_n):=\mathbb{P}(\omega)$.

Consider the open $\mathcal{U}_i=\{[z_0:\dots:z_n]\in
\mathbb{P}(\omega_0,\dots,\omega_n);\ z_i\neq0 \}\subset
\mathbb{P}(\omega_0,\dots,\omega_n)$, with $i=0,1,\dots,n.$ Let
$\mu_{\omega_i}\subset \mathbb{C}^*$ be the subgroup of
$\omega_i$-th roots of unity. We can define the homeomorphisms
$\phi_i:\mathcal{U}_i\longrightarrow \mathbb{C}^{n}/\mu_{\omega_i}$,
by
$$
\phi_i([z_0:\dots:z_n])=\left(\frac{z_0}{z_i^{\omega_0/\omega_i,}},
\dots,\frac{\widehat{z_i}}{z_i},\dots,\frac{z_n}{z_i^{\omega_n/\omega_i,}}
\right)_{\omega_i},
$$
where the symbol $``\;\widehat{\;}\;"$ means omission and
$(\cdot)_{\omega_i}$ is a $\omega_i$-conjugacy class in
$\mathbb{C}^{n}/\mu_{\omega_i}$ with $\mu_{\omega_i}$ acting on
$\mathbb{C}^{n}$ by
$$
\lambda\cdot(z_0,\dots, \hat z_i,\dots,z_n)=(\lambda^{\omega_0} z_0,\dots
\widehat{z_i},\dots,\lambda^{\omega_n} z_n), \lambda \in
\mu_{\omega_i}.
$$
On $\phi_i(\mathcal{U}_i\cap\mathcal{U}_j)\subset
\mathbb{C}^{n}/\mu_{\omega_i}$ we have the transitions maps  $(j<i$)
$$
\phi_i\circ
\phi_j^{-1}((z_0,\dots,\hat z_i, \dots, z_n)_{\omega_i})=\left(\frac{z_0}{z_j^{\omega_0/\omega_j,}},\dots,\frac{\widehat{z_j}}{z_j},
\dots,\frac{1}{z_j^{\omega_i/\omega_j}},\dots,\frac{z_n}{z_j^{\omega_n/\omega_j}}
\right)_{\omega_j}.
$$

\subsection{Line bundles on
$\mathbb{P}(\omega)$ and quasi-homogeneous hypersurface}

Let $d \in \mathbb{Z}$. Consider the $\mathbb{C}^*$-action
$\zeta_{d}$ on $\mathbb{C}^{n+1}\backslash
\{0\}\times \mathbb{C}$ given by
$$
\begin{array}{ccc}
   \zeta_{d}:\mathbb{C}^*\times\mathbb{C}^{n+1}\backslash
   \{0\}\times \mathbb{C}& \longrightarrow & \mathbb{C}^{n+1}\backslash \{0\}\times \mathbb{C} \\
  (\lambda,(z_0,\dots,z_n),t) & \longmapsto & ((\lambda^{\omega_0}
z_0,\dots,\lambda^{\omega_n} z_n),\lambda^{d}t).
\end{array}
$$
We denote  the quotient space induced by the action
$\zeta_{d}$ by
$$\mathcal{O}_{\mathbb{P}(\omega)}(d):=(\mathbb{C}^{n+1}\backslash
\{0\}\times \mathbb{C})/\sim \zeta_{d}.$$ The
space $\mathcal{O}_{\mathbb{P}(\omega)}(d)$ is a line orbibundle
on $\mathbb{P}(\omega)$. It is possible to show that the
Picard group of $\mathbb{P}(\omega)$ is generated by
$\mathcal{O}_{\mathbb{P}(\omega)}(1)$, i.e. %that is
$$
Pic(\mathbb{P}(\omega)) :=\mathbb{Z}\cdot
 \mathcal{O}_{\mathbb{P}(\omega)}(1).
$$
As usual we set $\mathcal{O}_{\mathbb{P}(\omega)}(1)^{\otimes d}:=\mathcal{O}_{\mathbb{P}(\omega)}(d)$ for $d \in \mathbb{Z}$. Moreover, we have the identification $(d\geq 0)$
$$
\mathrm{H}^0(\mathbb{P}(\omega),\mathcal{O}_{\mathbb{P}(\omega)}(d))=
\bigoplus_{\omega_0k_0+\cdots+\omega_nk_n=d}\mathbb{C}\cdot(z_0^{k_1}\cdots
z_n^{k_n}).
$$
%That is
Thus, the global sections of
$\mathcal{O}_{\mathbb{P}(\omega)}(d)$ can be identify, in
homogeneous coordinates, with quasi-homogeneous polynomials of
degree equal to $d$.

Thus, a quasi-homogeneous hypersurface  $V$ on $\mathbb{P}(\omega)$,
of  quasi-homogeneity  degree $d_0$, is given by $V=\{f=0\}$, where
$f\in
\mathrm{H}^0(\mathbb{P}(\omega),\mathcal{O}_{\mathbb{P}(\omega)}(d_0))$.
We say that  $V=\{f=0\}$ is quasi-smooth if its tangent cone
$\{f=0\}$ on $\mathbb{C}^{n+1}-\{0\}$ is smooth.

\subsection{Foliations on $\mathbb{P}(\omega)$ and quasi-homogeneous vector fields}

A singular  one dimensional holomorphic foliation on $\mathbb{P}(\omega)$, of degree $d$, is given by an element of $\mathbb{P}\mathrm{H}^0(\mathbb{P}(\omega),T\mathbb{P}(\omega)\otimes\mathcal{O}_{\omega}(d-1))$.

On  $\mathbb{P}(\omega)$ we have an Euler sequence
$$
0\longrightarrow
\mathcal{O}_{\mathbb{P}(\omega)} \stackrel{\varsigma}{\longrightarrow}\bigoplus_{i=0}^{n
}\mathcal{O}_{\mathbb{P}(\omega)}(\omega_i) \longrightarrow
T\mathbb{P}(\omega) \longrightarrow 0,
$$
where $\mathcal{O}_{\mathbb{P}(\omega)}$ is the trivial line orbibundle  and  $T\mathbb{P}(\omega) = \mathrm{Hom} (\Omega_{\mathbb{P}(\omega)}^ 1,  \mathcal{O}_{\mathbb{P}(\omega)} )$ is the tangent  orbibundle of
$\mathbb{P}(\omega)$. The map $\varsigma$ is given explicitly by $
 \varsigma(1)=(\omega_0z_0,\dots,\omega_nz_n)$ (see \cite{So}).

Now,  let $X$ be a quasi-homogeneous  vector field of type $(\omega_0,
\dots, \omega_n)$ and degree $d$ on $\C^{n+1}$, i.e. %that is, writing
$X =\sum\limits_{i=0}^n P_i (z) \frac{\partial}{\partial z_i}$ %we have that
where each polynomial $P_i$ satisfies the ``weight-homogeneous'' relation
$$P_i (\lambda^{\omega_0}z_0, \dots, \lambda^{\omega_n}z_n) = \lambda^{d
+ \omega_i - 1}P_i (z_0,\dots, z_n), \ \ \forall i=1,\dots,n.$$
These
vector fields descend well to $\mathbb{P}(\omega)$. In fact,
tensoring the Euler sequence by $\OO_{\p_\omega}(d-1)$, we obtain
$$ 0 \longrightarrow
\OO_{\mathbb{P}(\omega)}(d-1) \longrightarrow
\bigoplus\limits_{i=0}^n \OO_{\mathbb{P}(\omega)}(d + \omega_i - 1)
\longrightarrow T \mathbb{P}(\omega) \otimes
\OO_{\mathbb{P}(\omega)}(d-1) \longrightarrow 0.
$$
It follows that a quasi-homogeneous vector field $X$ induces a
foliation $\mathcal{F}$ of $\mathbb{P}(\omega)$ and that $g\,R_\omega +
X$ define the same foliation as $X$, where $R_\omega$ is the adapted
radial vector field $R_\omega = \omega_0 z_0 \frac{\partial}{\partial z_0}
+\cdots+   + \omega_n z_n \frac{\partial}{\partial z_2}$, with $g$ a
quasi-homogeneous polynomial of  type $(\omega_0,\dots, \omega_n)$  and degree
$d-1$. Therefore,   a quasi-homogeneous  vector field of type $(\omega_0,
\dots, \omega_n)$ and degree $d$ on $\C^{n+1}$  induces a holomorphic foliation
on  $\mathbb{P}(\omega)$ given by a global section of
$\mathrm{H}^0(\mathbb{P}(\omega),T\mathbb{P}(\omega)\otimes\mathcal{O}_{\omega}(d-1)).$

We have the following condition on the degree of a foliation
$$
d>1-\max_{0 \leq i<j \leq n}\{\omega_{i}+\omega_{j} \}.
$$
In fact,  by  Bott's Formulae for weighted projective spaces
(see \cite{D}), we have that
\begin{center}$
\mathrm{H}^0(\mathbb{P}(\omega),T\mathbb{P}(\omega)\otimes\mathcal{O}_{\omega}(d-1))\simeq
\mathrm{H}^0(\mathbb{P}(\omega),
\Omega^{n-1}_{\mathbb{P}(\omega)}(\sum_{i=0}^{n} \omega_{i}+d-1))\neq
\emptyset$,
\end{center}
if and only if, $d-1>-\max\limits_{0 \leq i\neq j \leq n}\{\omega_{i}+\omega_{j} \}$.

An algebraic hypersurface $V\subset \mathbb{P}(\omega)$ in invariant by a foliation $\F$ if $T_p\F_p \subset T_pV$ for all $p\in V\backslash Sing(V)\cup Sing(\F)$.
\subsection{Orbifold Milnor numbers and Baum-Bott formula}

\begin{definition}\label{df14:definition 14}\rm
Let  $M$ be a complex orbifold and $\F$ a singular holomorphic
foliation on $M$. Let  $p \in M$ be and
$(\widetilde{U},G_{p},\varphi)$ an orbifold chart  $U$ of $p$, the
\emph{orbifold Minor number} of $\F$ on $p$ is the  rational number
$$\mu_{p}^{orb}(\F)=\frac{\mu_{\widetilde{p}}(\,\widetilde{\xi}\,)}{|G_p|},$$
where  $\mu_{\widetilde{p}}(\,\widetilde{\xi}\,)$ is the milnor
number of the local lift $\widetilde{\xi}$ on $\widetilde{p}$ of a
the vector field $\xi$ tangent to $\F$ on $\widetilde{U}/G_p.$
\end{definition}

We will use the following Baum-Bott theorem for orbifolds due to M.
Corr\^ea,  A. M. Rodr\'iguez, M. G. Soares \cite{CPM}.
\begin{theorem}\label{teor7}
Let  $M$ be a compact  complex orbifold, of dimension $n$, and $\F$
a singular holomorphic foliation on $M$ induced by a global section
of $TM\otimes L$, with isolated singularities. Then
$$\int_M^{orb} c_n(TM\otimes
L)=\sum_{p\in Sing(\mathcal{F})}\mu_{p}^{orb}(\F).$$

\end{theorem}

On Weighted projective spaces and quasi-homogeneous and quasi-smooth hypersurfaces we have the following.
\begin{corollary} \label{sing(F)}\cite{CPM} Let $\F$ be a foliation of degree $d$ on
$\mathbb{P}(\omega)$. Then
$$\sum_{p\in Sing(\mathcal{F})} \mu_{p}(\F)^{orb}=\frac{1}{\omega_{0}\cdot\cdot\cdot \omega_{n}}\sum_{j=0}^{n}(d-1)^{n-j}\sigma_{j}(\omega),$$
where $\sigma_k(\omega)$ denotes the $k$-th
elementary symmetric function.
\end{corollary}

\begin{corollary}\label{sing(V)} Let $V$ be a quasi-homogeneous and quasi-smooth hypersurface, of degree
$d_{0}$, invariant by a holomorphic foliations $\F$ of degree $d.$
Then
$$\sum_{p\in Sing(\mathcal{F})\cap V} \mu_{p}(\F)^{orb}=\frac{1}{\omega_{0}\cdots \omega_{n}}\sum_{j=0}^{n-1}\left[\sum_{k=0}^{j}(-1)^{k}\sigma_{j-k}(\omega)d_{0}^{\,k+1}\right](d-1)^{n-1-j},$$
where $\sigma_\ell(\omega)$ denotes the $\ell$-th
elementary symmetric function.
\end{corollary}
\begin{proof} It follows from theorem \ref{teor7} that
$$\sum_{p\in Sing(\mathcal{F})\cap V}
\mu_{p}(\F)^{orb}=\int_{V}^{orb} C_{n-1}(TV\otimes\mathcal{O}_{\omega}(d-1)|_{V}).$$
 In order to calculate this integral, consider the Euler sequence
%$\int_{V}C_{n-1}(TV\otimes\mathcal{O}_{\omega}(d-1)|_{V})$.
$$0 \longrightarrow TV \longrightarrow T\mathbb{P}^{n}_{\omega}|_{V} \longrightarrow N_{V} \longrightarrow 0\, ,$$
where $N_V$ is the normal bundle. Then
$$C(\mathbb{P}^{n}_{\omega})=C(V)C(N_V)$$
and
\begin{eqnarray}\label{eq:1.10}
C_{j}(V)=C_{j}(\mathbb{P}^{n}_{\omega})-C_{j-1}(V)C_{1}(N_V)\,,\,\, 1
\leq j \leq n-1.
\end{eqnarray}
Moreover, by using the Euler formula  we get
\begin{eqnarray}\label{eq:1.11}
C_{j}(\mathbb{P}^{n}_{\omega})=P_{j}(\omega)C_{1}(\mathcal{O}_{\omega}(1))^j.
\end{eqnarray}
On the other hand, since
$$N_{V}=\mathcal{O}_{\mathbb{P}^{n}_{\omega}}(V)|_{V}=\mathcal{O}_{\omega}(d_0)|_{V},$$
we have that
\begin{eqnarray}\label{eq:1.12}
C_{1}(N_V)=d_{0}C_{1}(\mathcal{O}_{\omega}(1)|_{V}),
\end{eqnarray}
and replacing (\ref{eq:1.11}) and  (\ref{eq:1.12}) in (\ref{eq:1.10}), we obtain
$$C_{j}(V)=\left(\sum_{k=0}^{j}(-1)^{k}\sigma_{j}(\omega)d_{0}^{\,k}\right)\mathcal{O}_{\omega}(1)^j\,,\,\, 0 \leq j \leq n-1.$$
Therefore
\begin{align}
 &\ \hspace{-1.5cm} C_{n-1}(TV\otimes\mathcal{O}_{\omega}(d-1)|_{V}) \nonumber\\
%&\mbox{\ \ \ }= C_{n-1}(V)+C_{n-2}(V)C_{1}(\mathcal{O}_{\omga}(d-1))+\cdots+C_{1}(\mathcal{O}_{\omga}(d-1))^{n-1} \nonumber\\
&=\sum_{j=0}^{n-1}C_{n-1-j}(V)C_{1}(\mathcal{O}_{\omega}(d-1))^{j}\nonumber\\
&= \sum_{j=0}^{n-1}C_{n-1-j}(V)C_{1}(\mathcal{O}_{\omega}(1))^j(d-1)^{j}\nonumber\\
%&\mbox{\ \ \ }= C_{n-1}(V)+C_{n-2}(V)C_{1}(\mathcal{O}_{\omega}(1))(d-1)+\cdots+C_{1}(\mathcal{O}_{\omega}(1))^{n-1}(d-1)^{n-1}  \nonumber\\
&=
\sum_{j=0}^{n-1}\left[\left(\sum_{k=0}^{j}(-1)^{k}\sigma_{j-k}(\omega)d_{0}^{\,k}\right)C_{1}(\mathcal{O}_{\omega}(1))^{n-1}\right](d-1)^{n-1-j}.\label{eq:13}
\end{align}
By Satake-Poincar\'e duality, since  $C_{1}(\mathcal{O}_{\omega}(d_0))$ is the  Satake-Poincar\'e dual of $V$, we obtain
\begin{align*}
\int_{V}^{orb} C_{1}(\mathcal{O}_{\omega}(1))^{n-1}
&=\int_{\mathbb{P}^{n}_{\omega}}^{orb} C_{1}(\mathcal{O}_{\omega}(1))^{n-1}\wedge C_{1}(\mathcal{O}_{\omega}(d_0))\\
&=d_{0}\int_{\mathbb{P}^{n}_{\omega}}^{orb} C_{1}(\mathcal{O}_{\omega}(1))^{n}=\frac{d_0}{\omega_{0}\cdots \omega_{n}}.\\
\end{align*}
The last integral is calculated in \cite{M} and \cite{CPM}. Finally, the result follows from  the equation (\ref{eq:13}), integrating on  the variety  $V$.
\end{proof}

\section{Proof of Theorem}

Denote by $d_0=\deg(V)$ and $d=\deg(\F)$. It follows from Corollary \ref{sing(V)} that
\begin{equation}\label{a}
\sum_{j=0}^{n-1}\left[\sum_{k=0}^{j}(-1)^{k}\sigma_{j-k}(w)d_{0}^{\,k+1}\right] (d-1)^{n-1-j} = (w_{0}\cdots w_{n})\sum_{p\,\in V /\, \xi|_{V}(p)=0}\mathcal{I}_{p}(\xi|_{V}) \geq 0.
\end{equation}

%$$\sum_{j=0}^{n-1}\left[\sum_{k=0}^{j}(-1)^{k}\sigma_{j-k}(w)d_{0}^{\,k+1}\right] (d-1)^{n-1-j} = (w_{0}\cdots w_{n})\sum_{p\,\in V /\, \xi|_{V}(p)=0}\mathcal{I}_{p}(\xi|_{V}) \geq 0.$$
Also, by  Corollary  \ref{sing(F)}, we have
\begin{equation}\label{b}
\sum_{k=0}^{n}(d-1)^{n-k}\sigma_{k}(w)-\sum_{j=0}^{n-1}\left[\sum_{k=0}^{j}(-1)^{k}\sigma_{j-k}(w)\,d_{0}^{\,k+1}\right] (d-1)^{n-1-j} 
\end{equation}
\begin{equation*}
= (w_{0}\cdots w_{n})\left(\sum_{p\,\in \mathbb{P}^{n}_{w} /\, \xi(p)=0}\mathcal{I}_{p}(\xi)
-\sum_{p\,\in V /\, \xi|_{V}(p)=0}\mathcal{I}_{p}(\xi|_{V})\right)\geq 0.
\end{equation*}\\

Firstly, we prove for  $n=2$. It follows from equation (\ref{a}) that
$$-d_{0}^{2}+(\sigma_{1}(w)+d-1)d_0=(w_{0}w_{1}w_{2})\sum_{p\,\in V /\, \xi|_{V}(p)=0}\mathcal{I}_{p}(\xi|_{V})\,\geq 0.$$
Then $0<d_{0}\leq \sigma_{1}(w)+d-1$. We have that $V \cap Sing(\mathcal{F}) \neq \varnothing$.  In fact, suppose that $V \cap Sing(\mathcal{F}) =\varnothing$. In this case $V$
is a leaf of the foliation and we have the isomorphism of normal orbibundles $\mathcal{N}_{\F}|_V\simeq N_{V| \mathbb{P}^{2}_{\omega}}=\mathcal{O}_{\omega}(d_0)|_V$. Since the curve $V$ is a orbifold and $V \cap Sing(\mathcal{F}) =\varnothing$, it follows from \cite{CPM} that 
$\deg(\mathcal{N}_{\F}|_V)=0$.  But this is absurd since  
$$\deg(\mathcal{N}_{\F}|_V)=\deg(\mathcal{O}_{\omega}(d_0)|_V)= [\mathcal{O}_{\omega}(d_0)]\cdot V= [\mathcal{O}_{\omega}(d_0)]\cdot [\mathcal{O}_{\omega}(d_0)]= d_0^2\neq 0.$$
Observe that $V \cap Sing(\mathcal{F}) \neq \varnothing$ implied that $d_{0} \neq \sigma_{1}(w)+d-1$, %, since $V \cap Sing(\mathcal{F}) \neq \varnothing$. 
thus  $d_{0} \leq \sigma_{1}(w)+d-2$ and this proves the Theorem for $n=2.$\\

For $n\geq 3$, in order to prove that $d_0\leq d-1+\alpha_n\sigma_1(\omega),$
consider  the polynomial  $\Psi(t)\in\mathbb{Z}[t]$ defined by
$$\Psi(t)=\sum_{j=0}^{n-1}\left(\sum_{k=0}^{j}(-1)^{k}\sigma_{j-k}(\omega)\,t^{\,k+1} \right)(d-1)^{n-1-j},$$
and define $\Omega_n(t)$ as
$$
\Omega_{n}(t)= \begin{cases}
                 \frac{d}{d t}\left(\frac{\Psi(t)}{t}\right), &  \mbox{if $n$ is even}\\
             \\  \frac{d}{d t}\,\Psi(t)   ,                   & \mbox{if $n$ is odd}. \\
             \end{cases}
$$
We claim that $\Omega_n(t)\ne 0$ for all $t\ge 0$ and $d\gg 0$. In order  to prove that, for each $m$ positive integer, we define $$P_m(t)=\sum_{j=0}^m (-1)^j (m+1-j)t^{m-j}\quad\text{and}\quad
Q_m(t)=P_m(t)-P_{m-1}(t)%= (m+1)x^{m}-2\sum_{j=0}^{m-1} (-1)^{j} (m-j) x^{m-j-1}
$$
where  $P_0(t):=1$.

\begin{lemma}\label{Q_positivo}  Let $m$ be a positive integer. If $m$ is even then the polynomial
$Q_{m}(t)$
 is positive for all  $t\ge 0$.
\end{lemma}

\proof
The result is equivalent to prove that the polynomial
$$F(t)=(t+1)^2Q_m(t)=(m+1)t^{m+2}+2t^{m+1}-(m+1)t^m+2$$
is positive for all $t\ge 0$. Observe that $F(0)=2$ and $F(t)\ge 4$ for all $t\ge 1$, then it is enough to prove that $F(t)$ is positive in $[0,1]$.
Since
$$F'(t)=(m+1)t^{m-1}((m+2)t^2+2t-m),$$
the  critical points of $Q(t)$ are $t=-1$, $t=0$ and $t=\dfrac m{m+2}$. In addition
\begin{align*}
F\left(\frac m{m+2}\right)&=\left(\frac m{m+2}\right)^m\left((m+1)\left(\frac m{m+2}\right)^2+2\left(\frac m{m+2}\right)-(m+1)\right)+2\\
&=\left(\frac m{m+2}\right)^m \left(\frac {-m^2-4m-4}{(m+2)^2}\right)+2\\
&>\left(\frac {-m^2-4m-4}{(m+2)^2}\right)+2=\frac {4(m+1)}{(m+2)^2}>0.
\end{align*}
Therefore
$Q_m(t)\ge \dfrac {4(m+1)}{(m+2)^2(t+1)^2}>0$ for all $t\ge 0$.

\begin{lemma}\label{desig_simetricas}  For all  $k\ge 1$ we have
$\dfrac {\sigma_{k+1}(\omega)}{\sigma_k(\omega)}< \dfrac {\sigma_1(\omega)}{k+1}.$
\end{lemma}
\proof Let  $\check\omega_I$ be  the tuple
$(\omega_1,\dots,\omega_n)$ omitting the coordinates $\omega_i$'s
with $i\in I$. Observe that

\begin{align*}
\sigma_k(\omega)\sigma_1(\omega)&=\sum_{i=1}^n \omega_i\sigma_k(\omega)=\sum_{i=1}^n \omega_i\left( \omega_i \sigma_{k-1}(\check \omega_i)+\sigma_k(\check\omega_i)\right)\\
&=\sum_{i=1}^n  \omega_i^2 \sigma_{k-1}(\check \omega_i)+\sum_{i=1}^n  \omega_i\sigma_k(\check \omega_i))\\
&=\sum_{i=1}^n  \omega_i^2 \sigma_{k-1}(\check\omega_i)+(k+1) \sigma_{k+1}(\omega)).\\
\end{align*}
%where  $\widehat{\omega_i}$ denotes omission.
 Thus, we get that
$$\frac {\sigma_{k+1}(\omega)}{\sigma_k(\omega)}= \frac {\sigma_1(\omega)}{k+1}- \frac {\sum\limits_{i=1}^n  \omega_i^2 \sigma_{k-1}(\check \omega_i)}{(k+1)\sigma_k(\omega)}
<   \frac {\sigma_1(\omega)}{k+1}.\qed$$

\begin{proposition}\label{positividade}
Suppose that $d\ge \sigma_1(\omega)+1$, then the polynomial
$(-1)^{n-1}\Omega_n(t)$ is positive for all $t\ge 0$.
\end{proposition}

\proof : 
First,  we consider the case when $n$ is odd. We have  that
\begin{align*}
\Omega_n(t)&=\sum_{j=0}^{n-1}\left(\sum_{k=0}^{j}(-1)^{k}\sigma_{j-k}(\omega)(k+1)t^{k} \right)(d-1)^{n-1-j}\\
&=\sum_{l=0}^{n-1} \sigma_l(\omega)\left(\sum_{k=0}^{n-1-l} (-1)^k(k+1) t^k (d-1)^{n-1-k-l}\right)\\
&=\sum_{l=0}^{n-1} \sigma_l(\omega)(d-1)^{n-1-l}\left(\sum_{k=0}^{n-1-l} (-1)^k (k+1)s^k\right)\\
&=\sum_{l=0}^{n-1} \sigma_l(\omega)(d-1)^{n-1-l}(-1)^{n-1-l}P_{n-1-l}(s)\\
\end{align*}
where $s=\frac t{d-1}$.

%(the case $n$ even is similar).  
We can write 
$\Omega_n(t)-\sigma_{n-1}(\omega)$  as
$$= \sum_{j=0}^{\frac {n-3}2}(d-1)^{n-2-2j}\left(
\sigma_{2j}(\omega)(d-1)P_{n-1-2j}(s)-
\sigma_{2j+1}P_{n-2-2j}(s)\right).
$$
On the other hand, by Lemma \ref{desig_simetricas} and the
hypothesis,  we know that $$\sigma_{2l}(\omega)(d-1)\ge
\sigma_{2l}(\omega)\sigma_1(\omega)\ge \sigma_{2l+1}(\omega).$$ It
follows that
\begin{align*}
\sigma_{2j}(\omega)(d-1)P_{n-1-2j}(s)& -\sigma_{2j+1}P_{n-2-2j}(s)\ge \\
&\ge\sigma_{2j}(\omega)(d-1)P_{n-1-2j}(s)- \sigma_{2j+1}\left|P_{n-2-2j}(s)\right|\\
&\ge\sigma_{2j}(\omega)(d-1)\left(P_{n-1-2j}(s)-\left|P_{n-2-2j}(s)\right|\right)\\
&\ge\sigma_{2j}(\omega)(d-1)\min\left\{(n-2j)s^{n-1-2j}, Q_{n-1-2j}(s)\right\}\\
&>0.
\end{align*}
This last inequality follows from Lemma \ref{Q_positivo}. Therefore
$\Omega_n(t)> \sigma_{n-1}(\omega)$ for all $t\ge 0$.

In the case $n$ even,  similarly we obtain that
\begin{align*}
\Omega_n(t)& =\sum_{l=0}^{n-1} \sigma_l(\omega)(d-1)^{n-2-l}(-1)^{n-1-l}P_{n-2-l}(s)\\
&=
\sum_{j=0}^{\frac {n-2}2}(d-1)^{n-2-2j}\left(
-\sigma_{2j}(\omega)(d-1)P_{n-2-2j}(s)+
\sigma_{2j+1}P_{n-3-2j}(s)\right), 
\end{align*}
and by the  previous argument, we have that  each term of this sum is negative, therefore $\Omega_n(t)< 0$ for all $t\ge 0$.
\qed \\

Now, in order to  finish the proof of Theorem, we have two case to consider:

\subsection{$n$ odd:}
It follows from  Proposition \ref{positividade} that
$\Psi'(t)=\Omega_{n}(t)> 0$ for all $ t \in \mathbb{R}^{+}$. Then,
$\Psi$ is a increasing function and by equation (\ref{b}), $\Psi(d_0)\leq\sum\limits_{l=0}^{n}(d-1)^{n-l}\sigma_{l}(w)$.
We claim that
$\Psi\left(d-1+\alpha_n\sigma_1(\omega)\right)>
\sum\limits_{l=0}^{n}(d-1)^{n-l}\sigma_{l}(w)$. In fact,  since
\begin{align*}
\Psi(t)& =\sum_{j=0}^{n-1}\left(\sum_{k=0}^{j}(-1)^{k}\sigma_{j-k}(\omega)\,t^{k+1} \right)(d-1)^{n-1-j}\\
&=\sum_{l=0}^{n-1} \sigma_l(\omega)(d-1)^{n-l} \sum_{k=0}^{n-1-l}(-1)^k\left(\frac t{d-1}\right)^{k+1}\\
&=\frac t{d-1+t} \sum_{l=0}^{n-1} \sigma_l(\omega)(d-1)^{n-l}\left(1-\left(\frac {-t}{d-1}\right)^{n-l}\right)\\
&=\frac t{d-1+t} \sum_{l=0}^{n-1} \sigma_l(\omega)((d-1)^{n-l}-(-t)^{n-l}),
\end{align*}
then, putting  $t=d-1+\alpha_n\sigma_1(\omega)$%\quad \text{where} \quad\alpha=0.5437$$ 
 and using that $\sigma_{2k+1}(\omega)<\frac 1{2k}\cdot  \sigma_{2k}(\omega)\sigma_1(\omega)$ for each $k\ge 1$, we have that
\begin{align*}
-\Psi(t)+&\sum\limits_{l=0}^{n}\sigma_{l}(w)(d-1)^{n-l}=\\
&= \sigma_n(\omega) +\frac 1{d-1+t} \sum_{l=0}^{n-1} \sigma_l(\omega)((d-1)^{n-l+1}-(-t)^{n-l+1})\\
&=\sigma_n(\omega)+\sigma_{n-1}(\omega)\frac {(d-1)^2-t^2}{d-1+t}+\\
&\hspace{1cm}+\frac 1{d-1+t} \sum_{k=0}^{\frac{n-3}2} \left( \sigma_{2k}(\omega)((d-1)^{n-2k+1}-t^{n-2k+1}) \right.\\
&\hspace{55mm}\left.+
\sigma_{2k+1}(\omega)((d-1)^{n-2k}+t^{n-2k})\right)\\
%&< \sigma_n(\omega)-\sigma_{n-1}(\omega)\sigma_1(\omega)\\
&<\frac 1{d-1+t} \sum_{k=1}^{\frac{n-3}2} \sigma_{2k}(\omega)\left( (d-1)^{n-2k}\left(d-1+\frac 1{2k}\sigma_1(\omega)\right)-t^{n-2k}\left(t-\frac 1{2k}\sigma_1(\omega)\right)\right)\\
&\hspace{4cm}
+(d-1)^{n+1}-t^{n+1}+\sigma_{1}(\omega)(d-1)^{n}+\sigma_1(\omega)t^{n}.
\end{align*}
Then, in order to conclude the proof, it is enough to show that each term of this summation is less or equal  to zero, or equivalently,  putting $s=d-1$, $\sigma=\sigma_1(\omega)$, we have to prove that
\begin{enumerate}[(I)]
\item %$(d-1)^{n-2k}\left(d-1+\frac 1{2k}\sigma_1(\omega)\right)-t^{n-2k}\left(t-\frac 1{2k}\sigma_1(\omega)\right)\le 0$ 
$\left(\dfrac t{s}\right)^{n-2k}\ge \dfrac{s+\frac 1{2k}\sigma}{t-\frac 1{2k}\sigma}$
for each $k=1,\dots, \frac {n-3}2$  and
\item  %$(d-1)^{n+1}-t^{n+1}+\sigma_{1}(\omega)(d-1)^{n}+\sigma_1(\omega)t^{n}\le 0$.
$t^n(t-\sigma)\ge s^n(s+\sigma)$.
\end{enumerate}

For item (I), the case $n=3$ is empty, then we can suppose that $n\ge 5$. By Bernoulli's inequality, we have
$$\left(\dfrac t{s}\right)^{n-2k}=\left(1+\frac{\alpha_n \sigma}{s}\right)^{n-2k}\ge 1+\frac{(n-2k)\alpha_n \sigma}{s},$$
so, it is enough to prove that right side of this inequality  is greater that 
$$\dfrac{s+\frac 1{2k}\sigma}{t-\frac 1{2k}\sigma}=1+\frac {(\frac 1k -\alpha_n)\sigma}{s+(\alpha_n-\frac 1{2k})\sigma}.$$
In fact, if $n\ge 7$, then $\alpha_n\ge \frac 2n$ and
\begin{align*}
(n-2k)\alpha_n&\left(s+\left(\alpha_n-\frac 1{2k}\right)\sigma\right)-\left(\frac 1k -\alpha_n\right)s\\
&>(n-2k)\alpha_n\left(s-\frac 1{2k}\sigma\right)-\left(\frac 1k -\alpha_n\right)s\\
&>\left((n-2k)\alpha_n\frac {2k-1}{2k}-\frac 1k +\alpha_n\right)s\\
&>\left(\frac n2\alpha_n-1\right)s>0,\\
\end{align*}
and when $n=5$, it follows that $k=1$ and
$$(n-2k)\alpha_n\left(s+\left(\alpha_n-\frac 1{2k}\right)\sigma\right)-\left(\frac 1k -\alpha_n\right)s>\left( \frac 52\alpha_5+3\alpha_5^2\right)s>0.$$

 Finally, we are going to prove (II). Making $U:=\frac s{\sigma}>1$, observe that
\begin{align*}
t^n(t-&\sigma)- s^n(s+\sigma)\\
&=(s+\alpha_n\sigma)^n(s+(\alpha_n-1)\sigma)-s^{n+1}-\sigma s^n\\
&=\sigma^{n+1}\left((U+\alpha_n)^n(U+\alpha_n-1)-U^{n+1}-U^n\right)\\
&=\sigma^{n+1}\left(((n+1)\alpha_n-2)U^n+\sum_{j=1}^{n+1} \left(\binom nj \alpha_n^j+\binom n{j-1} \alpha_n^{j-1}(\alpha_n-1)\right) U^{n-j+1}\right)\\
&=\sigma^{n+1}\left(((n+1)\alpha_n-2)U^n+\sum_{j=0}^{n+1} \left(\frac {n+1}j \alpha_n-1\right)  \binom n{j-1}\alpha_n^{j-1}U^{n-j+1}\right)
\end{align*}
Now, the polynomial $$F(X)=((n+1)\alpha_n-2)X^n+\sum\limits_{j=0}^{n+1} \left(\frac {n+1}j \alpha_n-1\right) \binom n{j-1} \alpha_n^{j-1}X^{n-j+1}$$ and its derivate $F'(X)$   satisfy  that
\begin{itemize}
\item  the  leading  coefficient  is positive,
\item the list of other coefficients is a decreasing sequence, and then it   only has one change of sign,
\item $F(0)<0$ and $F'(0)<0$,
\item $F(1)=(1+\alpha_n)^n\alpha_n-2=0$.
\end{itemize}
 Then, by Descartes' rule of signs, $F'(X)$ only has one positive root and that root  is in the interval $(0,1)$, thus $F(X)$ is an increasing function in $(1,\infty)$. Therefore
$$t^n(t-\sigma)- s^n(s+\sigma)=\sigma^{n+1}F(U)\ge \sigma^{n+1}F(1)=0,$$
as we want to prove. \qed

\subsection{$n$ even} By the equation (\ref{a}) we know that
$\Psi(d_0) \geq 0.$
Therefore, if we  define $\Phi (t) = \frac{\Psi(t)}{t}$, then $\Phi(d_0)\ge 0$ and from  Proposition \ref{positividade} we have that  $\Phi(t)$ is a decreasing function.

We claim that $\Phi(d-1+\alpha_n\sigma_1(\omega))<0$, and therefore $d_0< d-1+\alpha_n\sigma_1(\omega)$.
In fact, following the same procedure using previously, we have that
\begin{align*}
\Phi(t)&=\frac 1{d-1+t} \sum_{l=0}^{n-1} \sigma_l(\omega)((d-1)^{n-l}-(-t)^{n-l})\\
&=\frac 1{d-1+t} \sum_{k=0}^{\frac{n-2}2} \left( \sigma_{2k}(\omega)((d-1)^{n-2k+1}-t^{n-2k+1}) \right.\\
&\hspace{50mm}\left.+
\sigma_{2k+1}(\omega)((d-1)^{n-2k}+t^{n-2k})\right)
\end{align*}
From here, the same argument of the case odd works.\qed

%%%%%%%%%%%%%%%%%%%%%%%%%%%%%%%%%%%%%%%%%%%%%%%%%%%%%%%%%%%%%EXEMPLOS

\vskip 0.5cm

%%%%%%%%%%%%%%%%%%%%%%%%%%%%%%%%%%%%%%%%%%%%

\vskip 0.5cm

%%%%%%%%%%%%%%%%%%%%%%%%%%%%%%%%%%%%%%%%%%%%%%%%%%%%%%%%%%%%%

\

\

\end{document}